\documentclass[a4paper,english,oneside]{article}
\usepackage[T1]{fontenc}
\usepackage[utf8]{inputenc} 

\usepackage{geometry}
\usepackage{amsmath}
\usepackage{amsfonts}
\usepackage{amssymb}
\usepackage{amsthm}
\usepackage[english]{babel}
\usepackage{color}
\usepackage{lmodern}
\usepackage{pstricks-add}
\usepackage{hyperref}
\hypersetup{hidelinks}

\usepackage{mathscinet}
\usepackage{enumitem}
\usepackage[cal=boondoxo,bb=ams]{mathalfa}
\usepackage{microtype}

\usepackage{mathtools}
\usepackage{esint}

\theoremstyle{plain}
\newtheorem{theorem}{Theorem}[section]

\newtheorem{corollary}[theorem]{Corollary}

\theoremstyle{definition}

\theoremstyle{remark}
\newtheorem{remark}{Remark}

    \DeclareMathOperator\supp{supp}
    \DeclareMathOperator\meas{meas}

\begin{document}

\title{A note on the nonexistence of global solutions to the semilinear wave equation with nonlinearity of derivative-type in the generalized Einstein -- de Sitter spacetime}

\author{Makram Hamouda$\,^\mathrm{a}$, Mohamed Ali Hamza$\,^\mathrm{a}$, Alessandro Palmieri$\,^\mathrm{b}$}


\date{\small{ $\,^\mathrm{a}$ Basic Sciences Department, Deanship of Preparatory Year and Supporting Studies, P. O. Box 1982, Imam Abdulrahman Bin Faisal University, Dammam, KSA.} \\ 
\small{ $\,^\mathrm{b}$ Department of Mathematics, University of Pisa, Largo B. Pontecorvo 5, 56127 Pisa, Italy} \\ [2ex] \normalsize{\today} }
\maketitle

\begin{abstract}
In this paper, we establish blow-up results for the semilinear wave equation in generalized Einstein-de Sitter spacetime with nonlinearity of derivative type. Our approach is based on the integral representation formula for the solution to the corresponding linear problem in the one-dimensional case, that we will determine through Yagdjian's Integral Transform approach. As upper bound for the exponent of the nonlinear term, we discover a Glassey-type exponent which depends both on the space dimension and on the Lorentzian metric in the generalized Einstein-de Sitter spacetime.
\end{abstract}

\begin{flushleft}
\textbf{Keywords} Wave equation, Einstein-de Sitter spacetime, Integral transform, Blow-up, Glassey exponent, Nonlinearity of derivative type
\end{flushleft}

\begin{flushleft}
\textbf{AMS Classification (2010)} Primary: 35B44, 35C15, 35L71; Secondary: 35A08, 35B33, 35L15
\end{flushleft}

\section{Introduction}

The aim of the present paper is to establish a blow-up result for local in time solutions to the Cauchy problem with \emph{nonlinearity of derivative type} $|\partial_t u|^p$
\begin{align}\label{semilinear WE in EdS derivative type}
\begin{cases} \partial_t^2u-t^{-2k}\Delta u + 2t^{-1} \partial_t u=|\partial_t u|^p, &  x\in \mathbb{R}^n, \ t>1,\\
u(1,x)=\varepsilon u_0(x), & x\in \mathbb{R}^n, \\ u_t(1,x)=\varepsilon u_1(x), & x\in \mathbb{R}^n,
\end{cases}
\end{align} where $k\in (0,1)$, $p>1$ and $\varepsilon$ is a positive constant describing the size of Cauchy data.

In the related literature (see, for example, \cite{GalYag17EdS}), the differential operator with time-dependent coefficients on the left-hand side of \eqref{semilinear WE in EdS derivative type} is called the wave operator on the \emph{generalized Einstein-de Sitter spacetime}. This nomenclature is due to the fact that for $k=\frac{2}{3}$ and $n=3$ the operator $\partial_t^2-t^{-\frac{4}{3}}\Delta+2t^{-1}\partial_t$ coincides with the d'Alembertian operator in Einstein-de Sitter's Lorentzian metric.

In recent years, many papers have been devoted to the study of blow-up results and lifespan estimates for the semilinear wave equation in the generalized Einstein - de Sitter (EdS) spacetime with power nonlinearities \cite{GalYag17EdS,Pal20EdS} and generalizations \cite{TW20,TW20h,Pal20EdSmu}. More specifically, it has been conjectured that the critical exponent for the semilinear Cauchy problem with \emph{power nonlinearity} $|u|^p$
\begin{align}\label{semilinear WE in EdS power nonlinearity}
\begin{cases} \partial_t^2u-t^{-2k}\Delta u + 2t^{-1} \partial_t u=|u|^p, &  x\in \mathbb{R}^n, \ t>1,\\
u(1,x)=\varepsilon u_0(x), & x\in \mathbb{R}^n, \\ u_t(1,x)=\varepsilon u_1(x), & x\in \mathbb{R}^n,
\end{cases}
\end{align} is given by the competition between two exponents, namely,
\begin{align}
\max\left\{p_0(n,k), 1+\tfrac{2}{(1-k)n}\right\}, \label{crit exp u^p}
\end{align} where $p_0(n,k)$ is the greatest root of the quadratic equation
\begin{align*}
((1-k)n+1)(p-1)^2+((1-k)n-1-2k)(p-1)-4=0.
\end{align*} The results from the previously quoted papers have been obtained via four different approaches, namely, Kato's type lemma, iteration argument (together with a slicing procedure in the critical case), comparison argument, and test function method. We can summarize these results by saying that for local in time solutions to \eqref{semilinear WE in EdS power nonlinearity} it has been shown the validity of a blow-up result (under suitable sign and support assumptions for the Cauchy data) whenever the exponent of the power nonlinearity $|u|^p$ fulfills $1<p\leqslant \max\{p_0(n,k), 1+\tfrac{2}{(1-k)n}\}$. Due to these nonexistence results, it has been conjecture that the exponent in \eqref{crit exp u^p} is critical, even though the global existence of small data solutions is a completely open problem to the best of the authors' knowledge.

The goal of the present work is to establish a blow-up result for \eqref{semilinear WE in EdS derivative type} and to determine a candidate as critical exponent somehow related to the so-called \emph{Glassey exponent}
\begin{align}\label{def Glassey exponent}
p_{\mathrm{Gla}}(n)\doteq \begin{cases} \frac{n+1}{n-1} & \mbox{if} \ n>1, \\ \infty & \mbox{if} \ n=1, \end{cases}
\end{align}
 which is the critical exponent for the semilinear wave equation with nonlinearity of derivative type \cite{Joh81,Mas83,Sid83,Sch86,Ram87,Age91,HT95,Tzv98,Zhou01,HWY12}.

Our approach consists in a modification of Zhou's approach for the corresponding semilinear wave equation in \cite{Zhou01}. Consequently, we will devote the first part of the paper to the proof of an integral representation formula for the solution to the linear inhomogeneous Cauchy problem associated with \eqref{semilinear WE in EdS derivative type} in the case $n=1$ via Yagjian's Integral Transform approach. 

We consider also a second semilinear model, that can be studied analogously with the tools of the present work. If we delete the linear damping term $2t^{-1}\partial_t u$ in \eqref{semilinear WE in EdS derivative type}, we obtain somehow a semilinear \emph{Tricomi-type model with negative power in the speed of propagation} and nonlinearity of derivative type. Hence, with minor modifications in our blow-up argument for \eqref{semilinear WE in EdS derivative type} we will prove a blow-up result also for the following model
\begin{align}\label{semilinear Tricomi negative derivative type}
\begin{cases} \partial_t^2u-t^{-2k}\Delta u=|\partial_t u|^p, &  x\in \mathbb{R}^n, \ t>1,\\
u(1,x)=\varepsilon u_0(x), & x\in \mathbb{R}^n, \\ u_t(1,x)=\varepsilon u_1(x), & x\in \mathbb{R}^n,
\end{cases}
\end{align} where $k\in (0,1)$, $p>1$ and $\varepsilon$ is a positive constant describing the size of Cauchy data. 

Recently, the semilinear generalized Tricomi equation with nonlinearity of derivative type, namely, the Cauchy problem
\begin{align}\label{semilinear Tricomi derivative type}
\begin{cases} \partial_t^2u-t^{2\ell}\Delta u=|\partial_t u|^p, &  x\in \mathbb{R}^n, \ t>0,\\
u(0,x)=\varepsilon u_0(x), & x\in \mathbb{R}^n, \\ u_t(0,x)=\varepsilon u_1(x), & x\in \mathbb{R}^n,
\end{cases}
\end{align} where $\ell >0$, $p>1$, $\varepsilon>0$, has been attracting a lot of attention and different approaches have been applied to establish several blow-up results for $p$ in suitable ranges depending on $n,\ell$. Among these approaches we find, on the one hand, comparison arguments based either on the fundamental solution for the operator $\partial_t^2-t^{2\ell}\Delta$ in \cite{LP20} or on the employment of a special positive solution of the corresponding homogeneous equation involving modified Bessel functions of second kind in \cite{HH20Tri}, and, on the other hand, a \textit{modified test function method} in \cite{LS20}.
In particular, the result that we are going to show for \eqref{semilinear Tricomi negative derivative type} is consistent with those for \eqref{semilinear Tricomi derivative type} in \cite{LS20,HH20Tri}, although we shall use a different approach to prove this result. Moreover, we will see that the upper bound for the exponent $p$ in the blow-up result for \eqref{semilinear WE in EdS derivative type} is a shift of the corresponding upper bound for \eqref{semilinear Tricomi negative derivative type}.

\subsection{Main results}  \label{Subsection Main results}
In this section, we state the main blow-up results for the semilinear models in \eqref{semilinear WE in EdS derivative type} and \eqref{semilinear Tricomi negative derivative type}, respectively. 

Since the speed of propagation in both semilinear models \eqref{semilinear WE in EdS derivative type} and \eqref{semilinear Tricomi negative derivative type} is $a_k(t)\doteq t^{-k}$, in the following it will be useful to employ the notation
\begin{align}\label{def phi k}
\phi_k(t)\doteq \frac{t^{1-k}}{1-k},
\end{align} for the primitive of $a_k$ vanishing at $t=0$ and the notation
\begin{align} \label{def A k}
A_k(t)\doteq \phi_k(t)-\phi_k(1) =\int_1^t a_k(\tau) \, \mathrm{d}\tau,
\end{align} for the distance function describing the amplitude of the curved light-cone.

\begin{theorem}[Blow-up result for the semilinear wave equation in EdS spacetime] \label{Thm blow-up EdS} Let $n\geqslant 1$ and $k\in(0,1)$. We assume that  $(u_0,u_1)\in \mathcal{C}^2_0(\mathbb{R}^n)\times  \mathcal{C}^1_0(\mathbb{R}^n)$ are
 nonnegative functions and have supports contained in $B_R\doteq \{x\in\mathbb{R}^n: |x|<R\}$.
 Let us consider an exponent $p$ in the nonlinearity of derivative type such that 
\begin{align}\label{p condition Thm EdS}
1<p\leqslant p_{\mathrm{Gla}}\big((1-k)n+2k+2\big),
\end{align} where the Glassey exponent $p_{\mathrm{Gla}}$ is defined by \eqref{def Glassey exponent}.
 
 Then, there exists $\varepsilon_0=\varepsilon_0(n,p,k,u_0,u_1,R)>0$ such that for any $\varepsilon\in (0,\varepsilon_0]$ if $u\in \mathcal{C}^2([1,T)\times \mathbb{R}^n)$ is a local in time solution to \eqref{semilinear WE in EdS derivative type} such that $\supp u(t,\cdot) \subset B_{R+A_k(t)}$ for any $t\in [1,T)$
  and with lifespan $T=T(\varepsilon)$, then,  $u$ blows up in finite time, that is, $T<\infty$.

 Furthermore, the following upper bound estimate for the lifespan holds
 \begin{align} \label{upper bound lifespan Thm EdS}
T(\varepsilon) \leqslant
\begin{cases} 
C \varepsilon^{-\left(\frac{1}{p-1}-\frac{(1-k)n+2k+1}{2}\right)^{-1}} & \mbox{if} \ \ 1<p<p_{\mathrm{Gla}}\big((1-k)n+2k+2\big), \\
 \exp\big(C \varepsilon^{-(p-1)}\big) & \mbox{if} \ \ p=p_{\mathrm{Gla}}\big((1-k)n+2k+2),
\end{cases}
\end{align} where the positive constant $C$ is independent of $\varepsilon$.
\end{theorem}

\begin{theorem}[Blow-up result for the semilinear Tricomi-type equation] \label{Thm blow-up Tricomi negative} Let $n\geqslant 1$ and $k\in(0,1)$. We assume that  $u_0=0$ and  that $u_1\in  \mathcal{C}^1_0(\mathbb{R}^n)$ is
 nonnegative function with support contained in $B_R$ for some $R>0$.
 Let us consider an exponent $p$ in the nonlinearity of derivative type such that 
\begin{align}\label{p condition Tricomi negative}
1<p\leqslant p_{\mathrm{Gla}}\big((1-k)n+2k\big),
\end{align} where the Glassey exponent $p_{\mathrm{Gla}}$ is defined by \eqref{def Glassey exponent}.
 
 Then, there exists $\varepsilon_0=\varepsilon_0(n,p,k,u_1,R)>0$ such that for any $\varepsilon\in (0,\varepsilon_0]$ if $u\in \mathcal{C}^2([1,T)\times \mathbb{R}^n)$ is a local in time solution to \eqref{semilinear Tricomi negative derivative type} such that $\supp u(t,\cdot) \subset B_{R+A_k(t)}$ for any $t\in [1,T)$
  and with lifespan $T=T(\varepsilon)$, then,  $u$ blows up in finite time, that is, $T<\infty$.

 Furthermore, the following upper bound estimate for the lifespan holds
 \begin{align} \label{upper bound lifespan Thm Tricomi negative}
T(\varepsilon) \leqslant
\begin{cases} 
C \varepsilon^{-\left(\frac{1}{p-1}-\frac{(1-k)n+2k-1}{2}\right)^{-1}} & \mbox{if} \ \ 1<p<p_{\mathrm{Gla}}\big((1-k)n+2k\big), \\
 \exp\big(C \varepsilon^{-(p-1)}\big) & \mbox{if} \ \ p=p_{\mathrm{Gla}}\big((1-k)n+2k),
\end{cases}
\end{align} where the positive constant $C$ is independent of $\varepsilon$.
\end{theorem}

\begin{remark} The upper bound for $p$ in the blow-up range \eqref{p condition Thm EdS} is a shift of magnitude $2$ of the Glassey exponent that appears as upper bound in \eqref{p condition Tricomi negative}. This kind of phenomenon has already been observed in the semilinear model with power nonlinearity in \cite{TW20,Pal20EdSmu}.
\end{remark}

\begin{remark} The upper bound $p_{\mathrm{Gla}}\big((1-k)n+2k\big)$ in Theorem \ref{Thm blow-up Tricomi negative} is consistent with the upper bound for the semilinear generalized Tricomi with nonlinearity of derivative type (when the power in the speed of propagation is positive and the Cauchy data are assumed at the initial time $t=0$) see e.g. \cite{LS20,HH20Tri}.
\end{remark}

\begin{remark} In Theorem \ref{Thm blow-up Tricomi negative}  we required a trivial first Cauchy data ($u_0=0$). This assumption is due to the fact that, in general, the kernel function $K_0\big(t,x;y;0,0,k\big)$,  whose definition will be provided later, see \eqref{def K0(t,x;y)}, is not a nonnegative function.
\end{remark}

\section{Integral representation formula} \label{Section n=1}

In the series of papers \cite{Yag04,Yag06,Yag07,YagGal08,Yag09,YagGal09,Yag10,Yag12,Yag13,
Yag15,Yag15MN,Pal19RF,Yag20}, several integral representation formulae for solutions to Cauchy problems associated with linear hyperbolic equations with variable coefficients have been derived and applied both to study the necessity and the sufficiency part concerning the problem of  the global (in time) existence of solutions. The general scheme to determine an integral representation in the above cited literature is the following: the desired formula is obtained by considering the composition of two operators. On the one hand, the external operator is an integral transformation, whose kernel is determined by the time-dependent coefficients and/or by the lower-order terms in the associated partial differential operator. On the other hand, the internal operator is a \emph{solution operator for a family of parameter dependent Cauchy problems} (this step is often called, in the above quoted literature, a revised Duhamel's principle). In the special case in which the considered hyperbolic model is a wave equation with time-dependent speed of propagation and lower-order terms, the above mentioned solution operator associates with a given function the solution to the Cauchy problem for the classical free wave equation with the given function as first initial data and with vanishing second initial data.

In the present section, we are going to use Yagjian's Integral Transform approach in order to determine an explicit integral representation formula for the linear Cauchy problem
\begin{align}\label{linear inhomog CP 1d}
\begin{cases} \partial_t^2 u- t^{-2k} \partial_x^2 u +\mu \, t^{-1}\partial_t u+\nu^2 \, t^{-2} u=g(t,x), &  x\in \mathbb{R}, \ t>1,\\
u(1,x)=u_0(x), & x\in \mathbb{R}, \\ \partial_t u(1,x)=u_1(x), & x\in \mathbb{R},
\end{cases}
\end{align} where $\mu,\nu^2$ are nonnegative real parameters and $k\in(0,1)$. 
Of course, \eqref{linear inhomog CP 1d} represents a more general model than the ones which are necessary to prove the blow-up results for \eqref{semilinear WE in EdS derivative type} and \eqref{semilinear Tricomi negative derivative type} that are stated in Section \ref{Subsection Main results}. Nonetheless, thanks to the study of the representation formula for \eqref{linear inhomog CP 1d} we will be able to derive as special cases the representation formulae for the solutions to the corresponding linear wave equation in the Einstein-de Sitter spacetime and linear Tricomi-type equation with negative power in the speed of propagation. Although it might seem counterintuitive, the presence of the mass term $\nu^2 \, t^{-2} u$ simplifies the description of the symmetry of the second-order operator $$\mathcal{L}_{k,\mu,\nu^2}\doteq \partial_t^2 - t^{-2k} \partial_x^2+\mu \, t^{-1}\partial_t +\nu^2 \, t^{-2}$$ with respect to the parameter $\mu$. More in detail, the quantity
\begin{align}
\delta= \delta(\mu,\nu^2) \doteq (\mu-1)^2-4 \nu^2, \label{def delta}
\end{align} has a crucial role in determining some properties of the fundamental solution of $\mathcal{L}_{k,\mu,\nu^2}$. In the special case $k=0$ (the so-called wave operator with scale-invariant damping and mass), it is known in the literature that the value of $\delta$ affects not only the fundamental solution of $\mathcal{L}_{0,\mu,\nu^2}$ but also the critical exponents in the treatment of semilinear Cauchy problem associated with $\mathcal{L}_{0,\mu,\nu^2}$ with power nonlinearity \cite{NPR16,PalRei18,PalThesis,Pal18odd,Pal18even,PT18,DabbPal18,Dabb20}, nonlinearity of derivative type \cite{PalTu19,HH20der}, and combined nonlinearity \cite{HH20,HH20i,HH20m}.

We shall see that even in the case $k\in(0,1)$ some properties of the fundamental solution of $\mathcal{L}_{k,\mu,\nu^2}$ depend strongly on the value of $\delta$. Of course, in the general case, we will find an interplay of $\delta$ and $k$ in the description of the fundamental solution.

We point out that, even though in this section we will focus on the case $n=1$, analogously to what is done in \cite{Yag04,Yag07,YagGal08,YagGal09,Pal19RF} it is possible to extend the integral representation even to the higher dimensional case by using the spherical means and the method of descent.

Let us state now the representation formula for the solution to \eqref{linear inhomog CP 1d} in space dimension 1.

\begin{theorem} \label{Thm representation formula 1d case}
Let $n=1$, $k\in(0,1)$ and let $\mu,\nu^2$ be nonnegative constants. We assume $u_0\in \mathcal{C}^2(\mathbb{R})$, $u_1\in \mathcal{C}^1(\mathbb{R})$ and $g\in \mathcal{C}^{0,1}_{t,x}([1,\infty)\times \mathbb{R})$. Then, a representation formula for the solution $u$ to \eqref{linear inhomog CP 1d} is given by
\begin{align}
u(t,x) &= \frac{1}{2}t^{\frac{k-\mu}{2}}\big(u_0(x+A_k(t))+u_0(x-A_k(t))\big)+\int_{x-A_k(t)}^{x+A_k(t)} u_0(y) K_0\big(t,x;y;\mu,\nu^2,k\big)\, \mathrm{d}y \notag \\ & \quad +\int_{x-A_k(t)}^{x+A_k(t)} u_1(y)K_1\big(t,x;y;\mu,\nu^2,k\big)\, \mathrm{d}y + \int_0^t \int_{x-\phi_k(t)+\phi_k(b)}^{x+\phi_k(t)-\phi_k(b)} g(b,y) E\big(t,x;b,y;\mu,\nu^2,k\big) \, \mathrm{d}y\, \mathrm{d}b. \label{representation formula 1d case}
\end{align} Here the kernel function $E$ is defined by
\begin{align}
E\big(t,x;b,y;\mu,\nu^2,k\big) & \doteq c\,  t^{-\frac{\mu}{2}+\frac{1-\sqrt{\delta}}{2}} b^{\frac{\mu}{2}+\frac{1-\sqrt{\delta}}{2}} \left((\phi_k(t)+\phi_k(b))^2-(y-x)^2\right)^{-\gamma} \notag \\ 
& \qquad \times \mathsf{F}\left(\gamma,\gamma;1; \frac{(\phi_k(t)-\phi_k(b))^2-(y-x)^2}{(\phi_k(t)+\phi_k(b))^2-(y-x)^2} \right),\label{def E(t,x;b,y)}
\end{align} where
\begin{align}
 \ c &=c(\mu, \nu^2,k) \doteq 2^{-\frac{\sqrt{\delta}}{1-k}}(1-k)^{-1+\frac{\sqrt{\delta}}{1-k}} \quad \mbox{and} \quad \gamma = \gamma(\mu,\nu^2,k)\doteq \frac{1}{2}-\frac{\sqrt{\delta}}{2(1-k)}, \label{def gamma} 
\end{align}  and $\mathsf{F}(\alpha_1,\alpha_2;\beta; z)$ denotes the Gauss hypergeometric function, while the kernel functions $K_0,K_1$ appearing in the integral terms involving the Cauchy data are given by
\begin{align}
K_0\big(t,x;y;\mu,\nu^2,k\big) & \doteq \mu E\big(t,x;1,y;\mu,\nu^2,k\big)-\frac{\partial}{\partial b}\, E\big(t,x;b,y;\mu,\nu^2,k\big) \Big|_{b=1}, \label{def K0(t,x;y)}\\
K_1\big(t,x;y;\mu,\nu^2,k\big) & \doteq  E\big(t,x;1,y;\mu,\nu^2,k\big). \label{def K1(t,x;y)}
\end{align} 
\end{theorem}


\begin{proof}
We are going to prove the representation formula in \eqref{representation formula 1d case} by means of a suitable change of variables that transforms \eqref{linear inhomog CP 1d} in a linear wave equation with scale-invariant damping and mass terms and allows us to employ a result from \cite{Pal19RF}. More specifically,
 we perform the transformation
 \begin{align}\label{change of variables y, tau}
 \tau \doteq t^{1-k}-1, \quad z\doteq (1-k)x.
 \end{align} Setting $v(\tau,z)=u(t,x)$, by straightforward computations it follows that $u$ solves \eqref{linear inhomog CP 1d} if and only if $v$ is a solution to 
 \begin{align*}
 \begin{cases} 
 \partial_\tau^2 v- \partial_z^2 v +\frac{\mu-k}{1-k} (1+\tau)^{-1} \partial_\tau v+\frac{\nu^2}{(1-k)^2} (1+\tau)^{-2} v=f(\tau,z), &  z\in \mathbb{R}, \ \tau>0,\\
v(\tau=0,z)=u_0\big(\frac{z}{1-k}\big), & z\in \mathbb{R}, \\ \partial_\tau v(\tau=0,z)=\frac{1}{1-k} u_1\big(\frac{z}{1-k}\big), & z\in \mathbb{R},
\end{cases}
 \end{align*} where $f(\tau,z)\doteq (1-k)^{-2}(1+\tau)^{\frac{2k}{1-k}}g\big((1+\tau)^{\frac{1}{1-k}},\frac{z}{1-k}\big)$. 
 According to \cite[Theorem 1.1]{Pal19RF}, we can represent $v$ in the following way 
 \begin{align}\label{v sum vj}
 v=\sum_{j=0}^3 v_j,
\end{align} 
  where the addends $\{v_j\}_{j\in\{0,1,2,3\}}$ are given by 
 \begin{align*}
 v_0(\tau,z) &\doteq \tfrac{1}{2} (1+\tau)^{-\tfrac{\widetilde{\mu}}{2}}\big(u_0\big(\tfrac{z+\tau}{1-k}\big)+u_0\big(\tfrac{z-\tau}{1-k}\big)\big)  ,\\
 v_1(\tau,z) &\doteq 2^{-\sqrt{\widetilde{\delta}\hphantom{l}}} \int_{z-\tau}^{z+\tau} u_0\big(\tfrac{\widetilde{y}}{1-k}\big) \widetilde{K}_0\big(\tau,z;\widetilde{y};\widetilde{\mu},\widetilde{\nu}^2\big)\, \mathrm{d}\widetilde{y},\\
 v_2(\tau,z) &\doteq  2^{-\sqrt{\widetilde{\delta}\hphantom{l}}} \int_{z-\tau}^{z+\tau} \Big(\tfrac{1}{1-k}u_1\big(\tfrac{\widetilde{y}}{1-k}\big)+\tfrac{\mu-k}{1-k} u_0\big(\tfrac{\widetilde{y}}{1-k}\big) \Big)\widetilde{K}_1\big(\tau,z;\widetilde{y};\widetilde{\mu},\widetilde{\nu}^2\big)\,\mathrm{d}\widetilde{y} ,\\
 v_3(\tau,z) &\doteq  2^{-\sqrt{\widetilde{\delta}\hphantom{l}}}  \int_0^\tau \int_{z-\tau+\widetilde{b}}^{z+\tau-\widetilde{b}} f(\widetilde{b},\widetilde{y}) \widetilde{E}\big(\tau,z;\widetilde{b},\widetilde{y};\widetilde{\mu},\widetilde{\nu}^2\big)\, \mathrm{d}\widetilde{y} \, \mathrm{d}\widetilde{b} ,
 \end{align*} and the kernel functions are given by
 \begin{align*}
\widetilde{E}\big(\tau,z;\widetilde{b},\widetilde{y};\widetilde{\mu},\widetilde{\nu}^2\big) & \doteq (1+\tau)^{-\tfrac{\widetilde{\mu}}{2}+\tfrac{1-\sqrt{\widetilde{\delta}\hphantom{l}}}{2}}\big(1+\widetilde{b}\big)^{\tfrac{\widetilde{\mu}}{2}+\tfrac{1-\sqrt{\widetilde{\delta}\hphantom{l}}}{2}} \big((\tau+\widetilde{b}+2)^2-(\widetilde{y}-z)^2\big)^{\tfrac{\sqrt{\widetilde{\delta}\hphantom{l}}-1}{2}} \\
& \qquad \times \mathsf{F}\left(\tfrac{1-\sqrt{\widetilde{\delta}\hphantom{l}}}{2}, \tfrac{1-\sqrt{\widetilde{\delta}\hphantom{l}}}{2};1;\frac{(\tau-\widetilde{b})^2-(\widetilde{y}-z)^2}{(\tau+\widetilde{b}+2)^2-(\widetilde{y}-z)^2}\right), \\
\widetilde{K}_0\big(\tau,z;\widetilde{y};\widetilde{\mu},\widetilde{\nu}^2\big)&\doteq -\frac{\partial}{\partial \widetilde{b}} \widetilde{E}\big(\tau,z;\widetilde{b},\widetilde{y};\widetilde{\mu},\widetilde{\nu}^2\big) \Big|_{\widetilde{b}=0} ,\\
\widetilde{K}_1\big(\tau,z;\widetilde{y};\widetilde{\mu},\widetilde{\nu}^2\big) &\doteq \widetilde{E}\big(\tau,z;0,\widetilde{y};\widetilde{\mu},\widetilde{\nu}^2\big) ,
 \end{align*} where $\mathsf{F}(\alpha_1,\alpha_2;\beta;z)$ is the Gauss hypergeometric function and  $$\widetilde{\mu}\doteq \frac{\mu-k}{1-k}, \quad \widetilde{\nu}\doteq \frac{\nu}{1-k}, \quad  \widetilde{\delta}= \widetilde{\delta}(\mu,\nu^2,k)\doteq(\widetilde{\mu}-1)^2-4\widetilde{\nu}^2 = \frac{ \delta(\mu,\nu^2)}{(1-k)^{2}}.$$
In order to show the validity of \eqref{representation formula 1d case}, we will transform back each term in \eqref{v sum vj}  through \eqref{change of variables y, tau}.

Let us begin with the function $v_0$. Recalling the definition of the function $A_k$ in \eqref{def A k}, we can write immediately
\begin{align}\label{v0}
v_0(\tau,y)=\tfrac12 t^{\frac{k-\mu}{2}}\big(u_0(x+A_k(t))+u_0(x-A_k(t))\big).
\end{align} Let us deal with the term $v_3$. Using the explicit expression of $f$, we get
\begin{align*}
v_3(\tau,z) &= 2^{-\sqrt{\widetilde{\delta}\hphantom{l}}}  \int_0^\tau \int_{z-\tau+\widetilde{b}}^{z+\tau-\widetilde{b}} (1-k)^{-2}(1+\widetilde{b})^{\frac{2k}{1-k}}g\big((1+\widetilde{b})^{\frac{1}{1-k}},\tfrac{\widetilde{y}}{1-k}\big)\widetilde{E}\big(\tau,z;\widetilde{b},\widetilde{y};\widetilde{\mu},\widetilde{\nu}^2\big)\, \mathrm{d}\widetilde{y} \, \mathrm{d}\widetilde{b}.
\end{align*} Carrying out the change of variables $y=(1-k)^{-1}\widetilde{y}$, $b=\big(1+\widetilde{b}\big)^{\frac{1}{1-k}}$ and the transformation \eqref{change of variables y, tau} in the last integral, we arrive at
\begin{align}
v_3(\tau,z) &= 2^{-\frac{\sqrt{\delta}}{1-k}}  \int_1^t \int_{x-\phi_k(t)+\phi_k(b)}^{x+\phi_k(t)-\phi_k(b)}g(b,y) \, b^k  \widetilde{E}\big(t^{1-k}-1,(1-k)x;b^{1-k}-1,(1-k)y;\widetilde{\mu},\widetilde{\nu}^2\big)\, \mathrm{d}y \, \mathrm{d}b\notag \\
&=   \int_1^t \int_{x-\phi_k(t)+\phi_k(b)}^{x+\phi_k(t)-\phi_k(b)}g(b,y) \, E\big(t,x;b,y;\mu,\nu^2,k\big)\, \mathrm{d}y \, \mathrm{d}b,\label{v3}
\end{align} where in the second step we used the identity
\begin{align*}
  \widetilde{E}\big(t^{1-k}-1,(1-k)x;b^{1-k}-1,(1-k)y;\widetilde{\mu},\widetilde{\nu}^2\big) = 2^{\frac{\sqrt{\delta}}{1-k}}  b^{-k} E\big(t,x;b,y;\mu,\nu^2,k\big),
\end{align*} and the definition in \eqref{def E(t,x;b,y)}.

Finally, we deal with the functions $v_1,v_2$. Performing the change of variables $y=(1-k)^{-1}\widetilde{y}$ and using \eqref{change of variables y, tau}, we find
\begin{align*}
v_1(\tau,z) &\doteq 2^{-\frac{\sqrt{\delta}}{1-k}}  (1-k)  \int_{x-A_k(t)}^{x+A_k(t)} u_0(y)\widetilde{K}_0\big(t^{1-k}-1,(1-k)x;(1-k)y;\widetilde{\mu},\widetilde{\nu}^2\big)\, \mathrm{d}y,\\
 v_2(\tau,z) &\doteq  2^{-\frac{\sqrt{\delta}}{1-k}} \int_{x-A_k(t)}^{x+A_k(t)}  \big(u_1(y)+(\mu-k)u_0(y) \big)\widetilde{K}_1\big(t^{1-k}-1,(1-k)x;(1-k)y;\widetilde{\mu},\widetilde{\nu}^2\big)\,\mathrm{d}y.
\end{align*} By elementary computations we have
\begin{align*}
\widetilde{K}_0\big(t^{1-k}-1,(1-k)x;(1-k)y;\widetilde{\mu},\widetilde{\nu}^2\big) &= - 2^{\frac{\sqrt{\delta}}{1-k}}  \frac{\partial}{\partial \widetilde{b}} \Big(\big(1+\widetilde{b}\big)^{-\frac{k}{1-k}} E\big(t,x;\big(1+\widetilde{b}\big)^{\frac{1}{1-k}},y;\mu,\nu^2,k\big)\Big)\Big|_{\widetilde{b}=0}, \\
\widetilde{K}_1\big(t^{1-k}-1,(1-k)x;(1-k)y;\widetilde{\mu},\widetilde{\nu}^2\big) &=  2^{\frac{\sqrt{\delta}}{1-k}}  E\big(t,x;1,y;\mu,\nu^2,k\big).
\end{align*} Considering the transformation $b=\big(1+\widetilde{b}\big)^{\frac{1}{1-k}}$ and using the relation
\begin{align*}
 \bigg(\frac{\partial}{\partial \widetilde{b}} \, \bigg)_{\widetilde{b}=0} = \frac{1}{1-k}  \bigg(\frac{\partial}{\partial b} \, \bigg)_{b=1} ,
\end{align*} we obtain 
\begin{align*}
(1-k) \widetilde{K}_0 \big(t^{1-k}-1,(1-k)x&;(1-k)y;\widetilde{\mu},\widetilde{\nu}^2\big) \\ &= - 2^{\frac{\sqrt{\delta}}{1-k}}  \frac{\partial}{\partial b} \Big(b^{-k} E\big(t,x;b,y;\mu,\nu^2,k\big)\Big)\Big|_{b=1} \\
&=  2^{\frac{\sqrt{\delta}}{1-k}} k  E\big(t,x;1,y;\mu,\nu^2,k\big)- 2^{\frac{\sqrt{\delta}}{1-k}}  \frac{\partial}{\partial b} \Big( E\big(t,x;b,y;\mu,\nu^2,k\big)\Big)\Big|_{b=1}.
\end{align*} Combining the previous representations for $\widetilde{K}_0, \widetilde{K}_1$, we conclude
\begin{align}
v_1(\tau,z)+v_2(\tau,z) & =  \int_{x-A_k(t)}^{x+A_k(t)} u_0(y)\left(\mu  E\big(t,x;1,y;\mu,\nu^2,k\big)-  \frac{\partial}{\partial b} \Big( E\big(t,x;b,y;\mu,\nu^2,k\big)\Big)\Big|_{b=1}\right) \mathrm{d}y\notag \\
& \qquad + \int_{x-A_k(t)}^{x+A_k(t)} u_1(y)  E\big(t,x;1,y;\mu,\nu^2,k\big)\, \mathrm{d}y. \label{v1+v2}
\end{align}
Summarizing, from \eqref{v sum vj}, \eqref{v0}, \eqref{v3} and \eqref{v1+v2} it follows immediately \eqref{representation formula 1d case}. This completes the proof.
\end{proof}

\begin{corollary}[Representation formula in EdS spacetime] \label{Cor representation formula 1d case EdS}
Let $n=1$ and $k\in(0,1)$. We assume $u_0\in \mathcal{C}^2(\mathbb{R})$, $u_1\in \mathcal{C}^1(\mathbb{R})$ and $g\in \mathcal{C}^{0,1}_{t,x}([1,\infty)\times \mathbb{R})$. Then, a representation formula for the solution $u$ to the Cauchy problem associated with the linear wave equation in Einstein-de Sitter spacetime
\begin{align}\label{linear inhomog CP 1d EdS}
\begin{cases} \partial_t^2 u- t^{-2k} \partial_x^2 u +2 \, t^{-1}\partial_t u=g(t,x), &  x\in \mathbb{R}, \ t>1,\\
u(1,x)=u_0(x), & x\in \mathbb{R}, \\ \partial_t u(1,x)=u_1(x), & x\in \mathbb{R},
\end{cases}
\end{align}
 is given by
\begin{align}
u(t,x) &= \frac{1}{2}t^{\frac{k}{2}-1}\big(u_0(x+A_k(t))+u_0(x-A_k(t))\big)+\int_{x-A_k(t)}^{x+A_k(t)} u_0(y) K_0\big(t,x;y;2,0,k\big)\, \mathrm{d}y \notag \\ & \quad +\int_{x-A_k(t)}^{x+A_k(t)} u_1(y)K_1\big(t,x;y;2,0,k\big)\, \mathrm{d}y  + \int_0^t \int_{x-\phi_k(t)+\phi_k(b)}^{x+\phi_k(t)-\phi_k(b)} g(b,y) E\big(t,x;b,y;2,0,k\big) \, \mathrm{d}y\, \mathrm{d}b. \label{representation formula 1d case EdS}
\end{align} Here the kernel function $E$ is defined by
\begin{align}
E\big(t,x;b,y;2,0,k\big) & \doteq c \,  t^{-1} b \left((\phi_k(t)+\phi_k(b))^2-(y-x)^2\right)^{-\gamma}  \mathsf{F}\left(\gamma,\gamma;1; \frac{(\phi_k(t)-\phi_k(b))^2-(y-x)^2}{(\phi_k(t)+\phi_k(b))^2-(y-x)^2} \right),\label{def E(t,x;b,y) EdS}
\end{align} where
\begin{align}
  c &=c(2,0,k) \doteq 2^{-\frac{1}{1-k}}(1-k)^{\frac{k}{1-k}} \quad \mbox{and} \quad \gamma = \gamma(2,0,k)\doteq -\frac{k}{2(1-k)}, \label{def gamma mu=2} 
\end{align}  and $\mathsf{F}(\alpha_1,\alpha_2;\beta; z)$ denotes the Gauss hypergeometric function, while the kernel functions $K_0,K_1$ appearing in the integral terms involving the Cauchy data are given by \eqref{def K0(t,x;y)} and \eqref{def K1(t,x;y)}, respectively, for the special values $(\mu,\nu^2)=(2,0)$.
\end{corollary}

\begin{corollary} [Representation formula for the Tricomi-type equation] \label{Cor representation formula 1d case Tricomi neg}
Let $n=1$ and $k\in(0,1)$. We assume $u_0\in \mathcal{C}^2(\mathbb{R})$, $u_1\in \mathcal{C}^1(\mathbb{R})$ and $g\in \mathcal{C}^{0,1}_{t,x}([1,\infty)\times \mathbb{R})$. Then, a representation formula for the solution $u$ to the Cauchy problem associated with the linear Tricomi-type equation
\begin{align}\label{linear inhomog CP 1d Tricomi neg}
\begin{cases} \partial_t^2 u- t^{-2k} \partial_x^2 u =g(t,x), &  x\in \mathbb{R}, \ t>1,\\
u(1,x)=u_0(x), & x\in \mathbb{R}, \\ \partial_t u(1,x)=u_1(x), & x\in \mathbb{R},
\end{cases}
\end{align}
 is given by
\begin{align}
u(t,x) &= \frac{1}{2}t^{\frac{k}{2}}\big(u_0(x+A_k(t))+u_0(x-A_k(t))\big)+\int_{x-A_k(t)}^{x+A_k(t)} u_0(y) K_0\big(t,x;y;0,0,k\big)\, \mathrm{d}y \notag \\ & \quad +\int_{x-A_k(t)}^{x+A_k(t)} u_1(y)K_1\big(t,x;y;0,0,k\big)\, \mathrm{d}y  + \int_0^t \int_{x-\phi_k(t)+\phi_k(b)}^{x+\phi_k(t)-\phi_k(b)} g(b,y) E\big(t,x;b,y;0,0,k\big) \, \mathrm{d}y\, \mathrm{d}b. \label{representation formula 1d case Tricomi neg}
\end{align} Here the kernel function $E$ is defined by
\begin{align}
E\big(t,x;b,y;0,0,k\big) & \doteq c \, \left((\phi_k(t)+\phi_k(b))^2-(y-x)^2\right)^{-\gamma}  \mathsf{F}\left(\gamma,\gamma;1; \frac{(\phi_k(t)-\phi_k(b))^2-(y-x)^2}{(\phi_k(t)+\phi_k(b))^2-(y-x)^2} \right),\label{def E(t,x;b,y) Tricomi neg}
\end{align} where
\begin{align}
  c &=c(0,0,k) \doteq 2^{-\frac{1}{1-k}}(1-k)^{\frac{k}{1-k}} \quad \mbox{and} \quad \gamma = \gamma(0,0,k)\doteq -\frac{k}{2(1-k)}, \label{def gamma mu=0} 
\end{align}  and $\mathsf{F}(\alpha_1,\alpha_2;\beta; z)$ denotes the Gauss hypergeometric function, while the kernel functions $K_0,K_1$ appearing in the integral terms involving the Cauchy data are given by \eqref{def K0(t,x;y)} and \eqref{def K1(t,x;y)}, respectively, for the special values $(\mu,\nu^2)=(0,0)$.
\end{corollary}

\begin{remark} \label{Remark kernel functions E} Note that the kernel functions in \eqref{def E(t,x;b,y) EdS} and in \eqref{def E(t,x;b,y) Tricomi neg} are strongly related, since the following relation holds
\begin{align*}
E\big(t,x;b,y;2,0,k\big)= \frac{b}{t} \, E\big(t,x;b,y;0,0,k\big).
\end{align*}
\end{remark}

\begin{remark}
Let us stress that representation formulae for the solutions to the wave equation in Einstein-de Sitter spacetime and to the generalized Tricomi equation (even in the case with speed of propagation with negative power) have already been established in the literature (see for example \cite{GKY10} and \cite{Yag15}, respectively) when the Cauchy data are prescribed at the initial time $t=0$. However, since in the present work we prescribe the Cauchy data at the initial data $t=1$ and the considered models are not invariant by time translation (due to the presence of time-dependent coefficients), the representation formulae from Corollaries \ref{Cor representation formula 1d case EdS} and \ref{Cor representation formula 1d case Tricomi neg} are not redundant and will play a crucial role in the proof of Theorems \ref{Thm blow-up EdS} and \ref{Thm blow-up Tricomi negative}, respectively.
\end{remark}

\begin{remark} \label{Remark Hypg funct}
In the next sections we shall estimate from below the kernel function $E$ on suitable subsets of the forward light-cone. According to this purpose, the following lower bound estimate for the Gauss hypergeometric function is very helpful
\begin{align} \label{lb Guass Hyp funct}
\mathsf{F}(\alpha,\alpha;\beta; z)\geqslant 1,
\end{align} for any $z\in[0,1)$ and for $\alpha\in\mathbb{R},\beta>0$. The previous estimate is a direct consequence of the series expansion for $\mathsf{F}(\alpha,\alpha;\beta; z)$.  Furthermore, for $\alpha\in\mathbb{R},\beta>0$ such that $\beta-2\alpha>0$ we have that \eqref{lb Guass Hyp funct} is sharp since $\mathsf{F}(\alpha,\alpha;\beta; z)$ can be estimated from above by a positive constant independent of $z$ on $[0,1)$.
\end{remark}

\section{Proof of the main Theorems}

In the present section, we prove the main blow-up results by using a generalization of Zhou's blow-up argument on the characteristic line $A_k(t)-z=R$. In place of the d'Alembert's formula we shall employ the integral representation formulae from Theorems \ref{Thm blow-up EdS} and \ref{Thm blow-up Tricomi negative} obtained via Yagdjian's Integral Transform approach. The main steps in the proof are inspired by the computations in \cite{Zhou01,PalTu19,LP20}.

\subsection{Proof of Theorem \ref{Thm blow-up EdS}} \label{Subsection proof Thm Eds}

Let $u$ be a local (in time) solution to the Cauchy problem \eqref{semilinear WE in EdS derivative type}. In order to reduce our problem to a one-dimensional problem in space, we will introduce an auxiliary function which depends on the time variable and only on the first space variable. This step is achieved by integrating $u$ with respect to the remaining $(n-1)$ spatial variables. Thus, if we denote $x=(z,w)$ with $z\in \mathbb{R}$ and $w\in \mathbb{R}^{n-1}$, then, we deal with the function
\begin{align*}
\mathcal{U}(t,z)\doteq \int_{\mathbb{R}^{n-1}} u(t,z,w) \, \mathrm{d} w \qquad \mbox{for any} \ \  t>0, z\in\mathbb{R}.
\end{align*} 
 Hereafter, we consider just the case $n\geqslant 2$ for the sake of brevity and of readability. Nevertheless, it is possible to proceed in the exact same way for $n=1$, simply by working with $u$ instead of introducing $\mathcal{U}$. In order to prescribe the initial values of $\mathcal{U}$, we set
\begin{align*}
\mathcal{U}_0(z)\doteq \int_{\mathbb{R}^{n-1}} u_0(z,w) \, \mathrm{d} w , \quad \mathcal{U}_1(z)\doteq \int_{\mathbb{R}^{n-1}} u_1(z,w) \, \mathrm{d} w\qquad \mbox{for any} \ \  z\in\mathbb{R}.
\end{align*} According to the statement of Theorem \ref{Thm blow-up EdS} we require that $u_0,u_1$ are compactly supported with support contained in $B_R$. Hence, $\mathcal{U}_0,\mathcal{U}_1$ are compactly supported too, with supports contained in $(-R,R)$. Analogously, since $\supp u(t,\cdot)\subset B_{R+A_k(t)}$ for any $t\in (1,T)$ we get the following support condition for $\mathcal{U}$
\begin{align} \label{supp mathcal U}
\supp \mathcal{U}(t,\cdot)\subset (-(R+A_k(t)),R+A_k(t)) \qquad \mbox{for any} \ \  t\in(1,T).
\end{align}  
Therefore, $\mathcal{U}$ solves the following Cauchy problem
\begin{align*} 
\begin{cases}
\partial_t^2\mathcal{U}- t^{-2k}\partial_z^2\mathcal{U}+2t^{-1} \partial_t\mathcal{U}=\int_{\mathbb{R}^{n-1}} |\partial_t u(t,z,w)|^p \, \mathrm{d} w , &  z\in \mathbb{R} , \ t>1,\\
\mathcal{U}(1,z)= \varepsilon\, \mathcal{U}_0(z), & z\in \mathbb{R} , \\ \partial_t\mathcal{U}(1,z)= \varepsilon\, \mathcal{U}_1(z) , & z\in \mathbb{R}.
\end{cases}
\end{align*} By Corollary \ref{Cor representation formula 1d case EdS}, we obtain the representation
\begin{align*}
\mathcal{U}(t,z) &= 2^{-1}\varepsilon t^{\frac{k}{2}-1}\big( \mathcal{U}_0(z+A_k(t))+ \mathcal{U}_0(z-A_k(t))\big)+\varepsilon\int_{z-A_k(t)}^{z+A_k(t)}  \mathcal{U}_0(y) K_0\big(t,z;y;2,0,k\big)\, \mathrm{d}y \notag \\ & \quad +\varepsilon\int_{z-A_k(t)}^{z+A_k(t)}  \mathcal{U}_1(y)K_1\big(t,z;y;2,0,k\big)\, \mathrm{d}y  \notag \\ & \quad  + \int_0^t \int_{z-\phi_k(t)+\phi_k(b)}^{z+\phi_k(t)-\phi_k(b)} \int_{\mathbb{R}^{n-1}} |\partial_t u(b,y,w)|^p \, \mathrm{d} w \, E\big(t,z;b,y;2,0,k\big) \, \mathrm{d}y\, \mathrm{d}b,
\end{align*} where the kernel function $E$ is defined by \eqref{def E(t,x;b,y) EdS}, while $K_0,K_1$ are defined in \eqref{def K0(t,x;y)} and in \eqref{def K1(t,x;y)}, respectively, for $(\mu,\nu^2)=(2,0)$.

Thanks to the sign assumptions on $u_0,u_1$ it results that $\mathcal{U}_0,\mathcal{U}_1$ are nonnegative functions as well. 
Consequently, from the previous identity we get 
\begin{align*}
\mathcal{U}(t,z) &\geqslant  \varepsilon \int_{z-A_k(t)}^{z+A_k(t)}  \left( \mathcal{U}_0(y) K_0\big(t,z;y;2,0,k\big) + \mathcal{U}_1(y)K_1\big(t,z;y;2,0,k\big)\right) \, \mathrm{d}y \notag \\ & \quad  + \int_0^t \int_{z-\phi_k(t)+\phi_k(b)}^{z+\phi_k(t)-\phi_k(b)} \int_{\mathbb{R}^{n-1}} |\partial_t u(b,y,w)|^p \, \mathrm{d} w \, E\big(t,z;b,y;2,0,k\big) \, \mathrm{d}y\, \mathrm{d}b\doteq \varepsilon J(t,z)+I(t,z).
\end{align*}

We begin with the estimate of the integral $J(t,z)$ that involves the Cauchy data. 

For any $y\in[z-A_k(t),z+A_k(t)]$, since $\gamma(2,0,k)=-k/(2-2k)<0$ and $1-2\gamma(2,0,k) =1/(1-k)>0$ (this second condition implies in particular that the hypergeometric function in \eqref{def E(t,x;b,y) EdS} can be controlled from above and from below by positive constants, cf. Remark \ref{Remark Hypg funct}), it results
\begin{align}
K_1\big(t,z;y;2,0,k\big) &=E\big(t,z;1,y;2,0,k\big) \notag
\\ &= c t^{-1} \left((\phi_k(t)+\phi_k(1))^2-(y-z)^2\right)^{-\gamma} \mathsf{F}\left(\gamma,\gamma;1; \frac{(\phi_k(t)-\phi_k(1))^2-(y-z)^2}{(\phi_k(t)+\phi_k(1))^2-(y-z)^2} \right) \notag\\ 
& \gtrsim t^{-1} (4\phi_k(1)\phi_k(t))^{-\gamma}  \gtrsim t^{\frac{k}{2}-1}.\label{lower bound K1}
\end{align} Moreover, for any $y\in[z-A_k(t),z+A_k(t)]$ we may prove that
\begin{align}
K_0\big(t,z;y;2,0,k\big) & = 2E\big(t,z;1,y;2,0,k\big) -\frac{\partial}{\partial b} E\big(t,z;b,y;2,0,k\big)\Big|_{b=1} \gtrsim t^{\frac{k}{2}-1}. 
\label{nonnegativity K0}
\end{align}
Indeed, by straightforward computations we find
\begin{align*}
\frac{\partial}{\partial b} E\big(t,z;b,y;2,0,k\big) & =ct^{-1}  \left((\phi_k(t)+\phi_k(b))^2-(y-z)^2\right)^{-\gamma} \\ & \quad \times \left\{\left(1+\frac{k\phi_k(b)(\phi_k(t)+\phi_k(b))}{(\phi_k(t)+\phi_k(b))^2-(y-z)^2}\right)\mathsf{F}(\gamma,\gamma;1;\zeta)+\mathsf{F}(\gamma+1,\gamma+1;2;\zeta)\frac{\partial \zeta}{\partial b}\right\},
\end{align*} where $$\zeta=\zeta(t,z;b,y;k)\doteq  \frac{(\phi_k(t)-\phi_k(b))^2-(y-z)^2}{(\phi_k(t)+\phi_k(b))^2-(y-z)^2}.  $$ Due to 
\begin{align*}
\frac{\partial \zeta}{\partial b}(t,z;b,y;k) =-4b^{-k}\phi_k(t)  \frac{(\phi_k(t))^2-(\phi_k(b))^2-(y-z)^2}{[(\phi_k(t)+\phi_k(b))^2-(y-z)^2]^2},
\end{align*} 
it follows
\begin{align*}
\frac{\partial \zeta}{\partial b}(t,z;b,y;k)\Big|_{b=1} =-4\phi_k(t)  \frac{(\phi_k(t))^2-(\phi_k(1))^2-(y-z)^2}{[(\phi_k(t)+\phi_k(1))^2-(y-z)^2]^2} \leqslant -  \frac{8\phi_k(t)\phi_k(1)(\phi_k(t)-\phi_k(1))}{[(\phi_k(t)+\phi_k(1))^2-(y-z)^2]^2}  \leqslant 0,
\end{align*} for any $y\in [z-A_k(t),z+A_k(t)]$.
Therefore, we may neglect the influence of the term $\mathsf{F}(\gamma+1,\gamma+1;2;\zeta)$ when we estimate the kernel $K_0$ from below. Hence,
\begin{align}
K_0\big(t,z;y;2,0,k\big) & \geqslant  ct^{-1}  \left((\phi_k(t)+\phi_k(1))^2-(y-z)^2\right)^{-\gamma} \left(1-\frac{k\phi_k(1)(\phi_k(t)+\phi_k(1))}{(\phi_k(t)+\phi_k(1))^2-(y-z)^2}\right)\mathsf{F}(\gamma,\gamma;1;\zeta) \notag \\
& \geqslant  ct^{-1}  \left((\phi_k(t)+\phi_k(1))^2-(y-z)^2\right)^{-\gamma} \left(1-\tfrac{k}{4}(1-\zeta) \right)\mathsf{F}(\gamma,\gamma;1;\zeta)\notag  \\
& \quad -   \tfrac{ck}{(1-k)^2} t^{-1}  \left((\phi_k(t)+\phi_k(1))^2-(y-z)^2\right)^{-\gamma-1}\mathsf{F}(\gamma,\gamma;1;\zeta). \label{intermediate estimate K0}
\end{align} Using the following estimate
\begin{align*}
 \frac{k}{(1-k)^2}  \left((\phi_k(t)+\phi_k(1))^2-(y-z)^2\right)^{-1} \leqslant  \frac{k}{(1-k)^2}   \left(4 \phi_k(t)\phi_k(1)\right)^{-1} = \frac{k}{4t^{1-k}}\leqslant  \frac{k}{4},
\end{align*} for $y\in [z-A_k(t),z+A_k(t)]$ and $t\geqslant 1$, from \eqref{intermediate estimate K0} we derive
\begin{align*}
K_0\big(t,z;y;2,0,k\big) 
& \geqslant  c\left(1-\tfrac{k}{2}+\tfrac{k}{4} \zeta \right) t^{-1}  \left((\phi_k(t)+\phi_k(1))^2-(y-z)^2\right)^{-\gamma}\, \mathsf{F}(\gamma,\gamma;1;\zeta) \\
&   \geqslant  c\left(1-\tfrac{k}{2} \right) t^{-1}  \left((\phi_k(t)+\phi_k(1))^2-A_k(t)^2\right)^{-\gamma}\, \mathsf{F}(\gamma,\gamma;1;\zeta)\\
&  =  c\left(1-\tfrac{k}{2} \right) t^{-1}  \left(4 \phi_k(1)\phi_k(t)\right)^{-\gamma}\, \mathsf{F}(\gamma,\gamma;1;\zeta) \gtrsim t^{\frac{k}{2}-1} \mathsf{F}(\gamma,\gamma;1;\zeta),
\end{align*} which implies in turn \eqref{nonnegativity K0} thanks to \eqref{lb Guass Hyp funct}.

Therefore, combining \eqref{lower bound K1} and \eqref{nonnegativity K0}, we obtain
\begin{align*}
J(t,z)\gtrsim \varepsilon t^{\frac{k}{2}-1} \int_{z-A_k(t)}^{z+A_k(t)}  ( \mathcal{U}_0(y) + \mathcal{U}_1(y)) \, \mathrm{d}y.
\end{align*}
From now on, we will work on the characteristic line with equation $A_k(t)-z=R$ for $z\geqslant R$. For $z\geqslant R$ such that $A_k(t)-z=R$ it holds $$[-R,R]\subset [z-A_k(t),z+A_k(t)].$$ Consequently, using the support condition $\mathrm{supp}\, \mathcal{U}_0,\mathrm{supp}\, \mathcal{U}_1\subset (-R,R)$, we conclude
\begin{align}
J(t,z)\gtrsim \varepsilon t^{\frac{k}{2}-1} \int_\mathbb{R} ( \mathcal{U}_0(y) + \mathcal{U}_1(y)) \, \mathrm{d}y = \varepsilon t^{\frac{k}{2}-1} \|u_0+u_1\|_{L^1(\mathbb{R}^n)}, \label{lower bound J}
\end{align} where in the last step we used Fubini's theorem.

Now we estimate the term $I(t,z)$. Clearly, the following support condition  $$\supp \partial_t u(t,\cdot) \subset B_{R+A_k(t)},$$ holds for any $t\in(1,T)$ due the shape of the light-cone. With respect to the $w\in\mathbb{R}^{n-1}$ variable we can express the previous support condition as follows: $$\supp \partial_t u(t,z,\cdot)\subset \big\{w\in \mathbb{R}^{n-1}: |w|\leqslant \left((R+A_k(t))^2-z^2\right)^{1/2} \big\}, $$ for any $t\in(1,T)$ and any $z\in\mathbb{R}$. Then, from H\"older's inequality we derive the estimate
\begin{align*}
  |\partial_t \mathcal{U} (b,y)| & =\Big| \int_{\mathbb{R}^{n-1}} \partial_t u(b, y,w) \, \mathrm{d}w\Big| \\ &  \leqslant \bigg(\int_{\mathbb{R}^{n-1}} |\partial_t u(b, y,w)|^p \, \mathrm{d}w\bigg)^{\frac{1}{p}} \big(\meas_{n-1}\big(\supp \partial_t u(b,y,\cdot)\big)\big)^{1-\frac{1}{p}} \\ & \lesssim \left((R+A_k(b))^2-y^2\right)^{\frac{n-1}{2}\left(1-\frac{1}{p}\right)}\bigg(\int_{\mathbb{R}^{n-1}} |\partial_t u(b, y,w)|^p \, \mathrm{d}w\bigg)^{\frac{1}{p}} .
\end{align*}
Hereafter, the unexpressed multiplicative constants will depend on $n,k,R,p$. 
 From the previous inequality, we have
\begin{align*}
\int_{\mathbb{R}^{n-1}} |\partial_t u(b, y,w)|^p \, \mathrm{d}w  & \gtrsim  \left((R+A_k(b))^2-y^2\right)^{-\frac{n-1}{2}(p-1)} |\partial_t \mathcal{U} (b,y)|^p,
\end{align*}  which provides in turn the following estimate
\begin{align*}
I(t,z) &  \gtrsim  \int_1^t \int_{z-\phi_k(t)+\phi_k(b)}^{z+\phi_k(t)-\phi_k(b)} \left((R+A_k(b))^2-y^2\right)^{-\frac{n-1}{2}(p-1)} |\partial_t \mathcal{U} (b,y)|^p  E\big(t,z;b,y;2,0,k\big) \, \mathrm{d}y\, \mathrm{d}b \\ 
& =  \int_{z-A_k(t)}^{z+A_k(t)} \int_1^{\phi^{-1}_k(\phi_k(t)-|z-y|)} \left((R+A_k(b))^2-y^2\right)^{-\frac{n-1}{2}(p-1)} |\partial_t \mathcal{U}(b,y)|^p  E\big(t,z;b,y;2,0,k\big)  \, \mathrm{d}b \, \mathrm{d}y ,
\end{align*} where we applied Fubini's theorem in the last equality. 
Since we are working on the  characteristic line $A_k(t)-z=R$, for $z\geqslant R$ it holds $[z-A_k(t),z+A_k(t)]\supset [R,z]$. Consequently, we can shrink the domain of integration in the last lower bound estimate for $I(t,z)$, obtaining 
\begin{align*}
I(t,z) &  \gtrsim   \int_{R}^{z} \int_1^{\phi^{-1}_k(\phi_k(t)-|z-y|)} \left((R+A_k(b))^2-y^2\right)^{-\frac{n-1}{2}(p-1)} |\partial_t \mathcal{U}(b,y)|^p  E\big(t,z;b,y;2,0,k\big)  \, \mathrm{d}b \, \mathrm{d}y.
\end{align*}
On the characteristic line $A_k(t)-z=R$ and for $y\in[R,z]$, it holds
\begin{align*}
\phi^{-1}_k(\phi_k(t)-|z-y|) = \phi^{-1}_k(\phi_k(t)-z+y)=  \phi^{-1}_k(\phi_k(1)+R+y) = A_k^{-1}(R+y),
\end{align*} where in  the last step we used the relation $$ A_k^{-1}(\tau)=\phi^{-1}_k(\phi_k(1)+\tau).$$ Moreover, for $y\in [R,z]$ it holds
\begin{align*}
1\leqslant  \phi^{-1}_k(\phi_k(1)+y-R)=A_k^{-1}(y-R), 
\end{align*} thanks to the monotonicity of $\phi^{-1}_k$. Thus, after a further shrinking of the domain of integration, we get
\begin{align*}
I(t,z) &  \gtrsim   \int_{R}^{z} \int_{A_k^{-1}(y-R)}^{A_k^{-1}(y+R)} \left((R+A_k(b))^2-y^2\right)^{-\frac{n-1}{2}(p-1)} |\partial_t \mathcal{U}(b,y)|^p  E\big(t,z;b,y;2,0,k\big)  \, \mathrm{d}b \, \mathrm{d}y \\
&  \gtrsim   \int_{R}^{z} \left((R+A_k(A_k^{-1}(y+R)))^2-y^2\right)^{-\frac{n-1}{2}(p-1)} \int_{A_k^{-1}(y-R)}^{A_k^{-1}(y+R)}  |\partial_t \mathcal{U}(b,y)|^p  E\big(t,z;b,y;2,0,k\big)  \, \mathrm{d}b \, \mathrm{d}y  \\
&  \gtrsim   \int_{R}^{z} (R+y)^{-\frac{n-1}{2}(p-1)} \int_{A_k^{-1}(y-R)}^{A_k^{-1}(y+R)}  |\partial_t \mathcal{U}(b,y)|^p  E\big(t,z;b,y;2,0,k\big)  \, \mathrm{d}b \, \mathrm{d}y,
\end{align*} for $z\geqslant R$ such that $A_k(t)-z=R$. 
Now we estimate the kernel function $E(\cdot;2,0,k)$ from below on the restricted domain of integration. Due to $1-2\gamma(2,0,k)>0$ the factor of $E$ involving the hypergeometric function can be controlled from above and from below with a positive constant. Consequently, for $y\in[R,z]$ and $b\in [A_k^{-1}(y-R),A_k^{-1}(y+R) ]$, on the characteristic $A_k(t)-z=R$ we have
\begin{align*}
 E\big(t,z;b,y;2,0,k\big) &= c \, \frac{b}{t}   \left((\phi_k(t)+\phi_k(b))^2-(y-z)^2\right)^{-\gamma} \mathsf{F}\left(\gamma,\gamma;1; \frac{(\phi_k(t)-\phi_k(b))^2-(y-z)^2}{(\phi_k(t)+\phi_k(b))^2-(y-z)^2} \right) \\
 &\gtrsim  \frac{b}{t}   \left((\phi_k(t)+\phi_k(b))^2-(y-z)^2\right)^{-\gamma} = \frac{b}{t}   \left((A_k(t)+A_k(b)+2\phi_k(1))^2-(y-z)^2\right)^{-\gamma} \\
 &\gtrsim  \frac{b}{t}   \left((A_k(t)-R+y+2\phi_k(1))^2-(y-z)^2\right)^{-\gamma} =  \frac{b}{t}   \left((z+y+2\phi_k(1))^2-(y-z)^2\right)^{-\gamma} \\ & =  \frac{b}{t}   \Big(4(z+\phi_k(1))(y+\phi_k(1))\Big)^{-\gamma} \gtrsim  \frac{A_k^{-1}(y-R)}{t}  (z+R)^{-\gamma}  (y+R)^{-\gamma},
\end{align*} where we used that $\gamma=\gamma(2,0,k)$ is a negative parameter. We remark that  for $y\in[R,z]$, it holds
\begin{align*}
 \frac{A_k^{-1}(y-R)}{A_k^{-1}(y+R)} & =\frac{\phi_k^{-1}(\phi_k(1)+y-R)}{\phi_k^{-1}(\phi_k(1)+y+R)}= \left(\frac{\phi_k(1)+y-R}{\phi_k(1)+y+R}\right)^{\frac{1}{1-k}}= \left(1-\frac{2R}{\phi_k(1)+y+R}\right)^{\frac{1}{1-k}} \\ & \geqslant  \left(1-\frac{2R}{\phi_k(1)+2R}\right)^{\frac{1}{1-k}}=C_{R,k}\,.
\end{align*} Combining the lower bound estimate for $E$ on the domain of integration and the last inequality, we arrive at 
\begin{align}
I(t,z) &  \gtrsim  (z+R)^{-\gamma}   t^{-1}  \int_{R}^{z} (R+y)^{-\frac{n-1}{2}(p-1)-\gamma} A_k^{-1}(y+R) \int_{A_k^{-1}(y-R)}^{A_k^{-1}(y+R)}  |\partial_t \mathcal{U}(b,y)|^p  \, \mathrm{d}b \, \mathrm{d}y \notag \\
&\gtrsim  (z+R)^{-\gamma}   t^{-1}  \int_{R}^{z} (R+y)^{-\frac{n-1}{2}(p-1)-\gamma+\frac{1}{1-k}} \int_{A_k^{-1}(y-R)}^{A_k^{-1}(y+R)}  |\partial_t \mathcal{U}(b,y)|^p  \, \mathrm{d}b \, \mathrm{d}y, \label{lower bound I n1}
\end{align} for $z\geqslant R$ such that $A_k(t)-z=R$.

Applying Jensen's inequality, we get
 \begin{align}
\left| \int_{A_k^{-1}(y-R)}^{A_k^{-1}(y+R)} \partial_t \mathcal{U}(b,y)   \, \mathrm{d}b \,\right|^p & \leqslant \big(A_k^{-1}(y+R)-A_k^{-1}(y-R)\big)^{p-1} \int_{A_k^{-1}(y-R)}^{A_k^{-1}(y+R)} |\mathcal{U}_t(b,y)|^p   \, \mathrm{d}b \notag \\
& \leqslant  (2R)^{p-1} \left(\max_{\tau\in [y-R,y+R]} \frac{\mathrm{d}}{\mathrm{d}\tau}A_k^{-1}(\tau)\right)^{p-1}\int_{A_k^{-1}(y-R)}^{A_k^{-1}(y+R)} |\partial_t \mathcal{U}(b,y)|^p   \, \mathrm{d}b \notag \\
& \lesssim  (\phi_k(1)+y+R)^{\frac{k}{1-k}(p-1)}\int_{A_k^{-1}(y-R)}^{A_k^{-1}(y+R)} |\partial_t \mathcal{U}(b,y)|^p   \, \mathrm{d}b, \label{Jensen mathcal U_t}
 \end{align} where we used $\frac{\mathrm{d}}{\mathrm{d}\tau}A_k^{-1}(\tau)=\big((1-k)(\phi_k(1)+\tau)\big)^{\frac{k}{1-k}}$.
 
 Combining the fundamental theorem of calculus, \eqref{lower bound I n1} and \eqref{Jensen mathcal U_t}, on the characteristic $A_k(t)-z=R$ we find
\begin{align}
I(A_k^{-1}(z+R),z)   
& \gtrsim (z+R)^{-\gamma -\frac{1}{1-k}} \! \int_{R}^{z} \!\left(R+y\right)^{-\frac{n-1}{2}(p-1)-\gamma+\frac{1}{1-k}-\frac{k}{1-k}(p-1)} \left| \int_{A_k^{-1}(y-R)}^{A_k^{-1}(y+R)}\!\partial_t \mathcal{U}(b,y)   \, \mathrm{d}b \right|^p \! \mathrm{d}y \notag \\
& = (z+R)^{-\gamma -\frac{1}{1-k}} \int_{R}^{z} \! \left(R+y\right)^{-\frac{n-1}{2}(p-1)-\gamma+\frac{1}{1-k}-\frac{k}{1-k}(p-1)} |\mathcal{U}(A_k^{-1}(y+R),y) |^p \, \mathrm{d}y,  \label{lower bound I n2}
\end{align} where in the second step we can apply the relation $\mathcal{U}(A_k^{-1}(y-R),y)=0$ due to \eqref{supp mathcal U}.

Using together \eqref{lower bound J} and \eqref{lower bound I n2}, on the characteristic line $A_k(t)-z=R$ for $z\geqslant R$, we obtain
\begin{align} 
(R+z)^{\gamma+\frac{1}{1-k}} &\mathcal{U}(A_k^{-1}(z+R),z)  \notag   \\ & \geqslant C \varepsilon \, \|  u_0+u_1 \|_{L^1(\mathbb{R}^n)} + C \int_{R}^{z} \left(R+y\right)^{-\frac{n-1}{2}(p-1)-\gamma+\frac{1}{1-k}-\frac{k}{1-k}(p-1)} |\mathcal{U}(A_k^{-1}(y+R),y) |^p \, \mathrm{d}y, \label{fundamental inequality for mathcal U}
\end{align} where $C=C(n,k,R,p)>0$ is a suitable multiplicative constant.

We are now able to define $$U(z)\doteq (R+z)^{\gamma+\frac{1}{1-k}} \mathcal{U}(A_k^{-1}(z+R),z), \quad z\geqslant R,$$ whose dynamic will be employed to prove the blow-up result. 

Rewriting  \eqref{fundamental inequality for mathcal U} through $U$, we find
\begin{align}\label{fundamental inequality for U}
U(z)\geqslant C \varepsilon \, \|  u_0+u_1 \|_{L^1(\mathbb{R}^n)}  + C \int_{R}^{z} \left(R+y\right)^{\left(-\frac{n-1}{2}-\frac{1+k}{1-k}-\gamma\right)(p-1)-2\gamma} |U(y) |^p \, \mathrm{d}y \qquad \mbox{for}  \ z\geqslant R.
\end{align} 
Finally, we apply a comparison argument to $U$. We define the auxiliary function $G$ in the following way 
\begin{align*}
G(z)\doteq   M \varepsilon +  C \int_{R}^{z} \left(R+y\right)^{\left(-\frac{n-1}{2}-\frac{1+k}{1-k}-\gamma\right)(p-1)-2\gamma} |U(y) |^p \, \mathrm{d}y  \qquad \mbox{for}  \ z\geqslant R,
\end{align*} where $M\doteq C \| u_0+u_1 \|_{L^1(\mathbb{R}^n)}$. From \eqref{fundamental inequality for U} we get immediately $U\geqslant G$. Furthermore, $G$ fulfills the ordinary differential inequality
\begin{align*}
G'(z) & = C \left(R+z\right)^{\left(-\frac{n-1}{2}-\frac{1+k}{1-k}-\gamma\right)(p-1)-2\gamma} |U(z) |^p \\
& \geqslant C  \left(R+z \right)^{\left(-\frac{n-1}{2}-\frac{1+k}{1-k}-\gamma\right)(p-1)-2\gamma}(G(z) )^p 
\end{align*} 
and satisfies the initial condition $G(R)=M\varepsilon$.
If $p$ satisfies
\begin{align} \label{critical expression}
\left(-\frac{n-1}{2}-\frac{1+k}{1-k}-\gamma\right)(p-1)-2\gamma=-1,
\end{align} then, 
\begin{align}\label{crit case G}
(M\varepsilon)^{1-p}-G(z)^{1-p}\geqslant C (p-1)\log \left(\frac{R+z}{2R}\right) .
\end{align}
We underline that \eqref{critical expression} is equivalent to  $p=p_{\mathrm{Gla}}\big((1-k)n+2k+2\big) $.

 Otherwise, since $G(z)>0$ for any $z\geqslant R$, it follows
\begin{align} 
(M\varepsilon)^{1-p} & -G(z)^{1-p} \notag \\ & \geqslant  \frac{C(1-k)}{\frac{1}{p-1}-\frac{(1-k)n+2k+1}{2}} \Big((R+z)^{1-2\gamma-\frac{(1-k)n+2k+1}{2(1-k)}(p-1)}-(2R)^{1-2\gamma-\frac{(1-k)n+2k+1}{2(1-k)}(p-1)}\Big). \label{subcritical expression}
\end{align}

Thus, if $p\in \big(1,p_{\mathrm{Gla}}\big((1-k)n+2k+2\big)\big)$, then, the multiplicative factor on the right-hand side of \eqref{subcritical expression} is positive. So, we let $\varepsilon_0=\varepsilon_0(n,p,k,u_0,u_1,R)$ sufficiently small such that for any $\varepsilon\in (0,\varepsilon_0]$ it results
\begin{align}
G(z) 
&\geqslant \left[2(M\varepsilon)^{1-p}-\frac{C(1-k)}{\frac{1}{p-1}-\frac{(1-k)n+2k+1}{2}}(R+z)^{1-2\gamma-\frac{(1-k)n+2k+1}{2(1-k)}(p-1)} \right]^{-\frac{1}{p-1}}. \label{estimate G below subcrit}
\end{align}

In the limit case $p=p_{\mathrm{Gla}}\big((1-k)n+2k+2\big)$, from \eqref{crit case G} we get the estimate
 \begin{align*}
U(z)\geqslant G(z) & \geqslant  \left[(M\varepsilon)^{1-p}-C(p-1)\log \left(\tfrac{R+z}{2R}\right)\right]^{-\frac{1}{p-1}},
 \end{align*} that provides the blow-up in finite time of $U(z)$  and the lifespan estimate $$\log T(\varepsilon)\lesssim \varepsilon^{-(p-1)}.$$
Otherwise, in the case $p\in \big(1,p_{\mathrm{Gla}}\big((1-k)n+2k+2\big)\big)$, the right-hand side of the inequality in \eqref{estimate G below subcrit} blows up for
$$A_k(t)=R+z\approx\varepsilon^{-(1-k)\left(\frac{1}{p-1}-\frac{(1-k)n+2k+1}{2}\right)^{-1}}.$$ Therefore, $G$ (and $U$, in turn) blows up and the following upper bound estimates holds 
 $$T(\varepsilon)\lesssim  \varepsilon^{-\left(\frac{1}{p-1}-\frac{(1-k)n+2k+1}{2}\right)^{-1}}.$$ 
This completes the proof of Theorem \ref{Thm blow-up EdS}.

\subsection{Proof of Theorem \ref{Thm blow-up Tricomi negative}}

The proof of Theorem \ref{Thm blow-up Tricomi negative} is analogous to the one of Theorem \ref{Thm blow-up EdS}. Let us just sketch the key points in the blow-up argument and emphasize the modifications that we have to carry out in comparison to the previous case. Given a local solution $u$ to \eqref{semilinear Tricomi negative derivative type}, we may introduce also in this case the function $\mathcal{U}=\mathcal{U}(t,z)$ for $t\geqslant 1$, $z\in\mathbb{R}$ by integrating with respect to the last $(n-1)$-space variables. Thanks to the assumption $u_0=0$ and Corollary \ref{Cor representation formula 1d case Tricomi neg}, the following representation holds
\begin{align*}
\mathcal{U}(t,z)& = \varepsilon \widetilde{J}(t,z)+\widetilde{I}(t,z),
\end{align*} where
\begin{align*}
 \widetilde{J}(t,z) & \doteq \int_{z-A_k(t)}^{z+A_k(t)}  \mathcal{U}_1(y)K_1\big(t,z;y;0,0,k\big)\, \mathrm{d}y, \\
 \widetilde{I}(t,z)   & \doteq \int_0^t \int_{z-\phi_k(t)+\phi_k(b)}^{z+\phi_k(t)-\phi_k(b)} \int_{\mathbb{R}^{n-1}} |\partial_t u(b,y,w)|^p \, \mathrm{d} w \, E\big(t,z;b,y;0,0,k\big) \, \mathrm{d}y\, \mathrm{d}b.
\end{align*} In order to estimate from below the terms $ \widetilde{J}, \widetilde{I}$, we can proceed very similarly as in the previous proof, keeping in mind the relation pointed out in Remark \ref{Remark kernel functions E} on the fundamental solution $E$. More precisely, in place of \eqref{lower bound J}, we get
\begin{align*}
 \widetilde{J}(t,z)\gtrsim  \varepsilon \, t^{\frac{k}{2}} \|u_1\|_{L^1(\mathbb{R}^n)},
\end{align*} while instead of \eqref{lower bound I n2}, we find
\begin{align*}
\widetilde{I}(A_k^{-1}(z+R),z)   
& \gtrsim (z+R)^{-\gamma }  \int_{R}^{z} \left(R+y\right)^{-\frac{n-1}{2}(p-1)-\gamma-\frac{k}{1-k}(p-1)} |\mathcal{U}(A_k^{-1}(y+R),y) |^p \, \mathrm{d}y, 
\end{align*}  on the characteristic $A_k(t)-z=R$, for $z\geqslant R$.
Therefore, the functional that we have to consider in order to study the blow-up dynamic is  
\begin{align*}
\widetilde{U}(z)\doteq   (z+R)^{\gamma } \mathcal{U}(A_k^{-1}(z+R),z).  
\end{align*} Consequently, the ordinary integral inequality for $\widetilde{U}$ is given by 
\begin{align*}
\widetilde{U}(z)\geqslant \widetilde{C} \varepsilon \, \|  u_1 \|_{L^1(\mathbb{R}^n)}  + \widetilde{C} \int_{R}^{z} \left(R+y\right)^{\left(-\frac{n-1}{2}-\frac{k}{1-k}-\gamma\right)(p-1)-2\gamma} |\widetilde{U}(y) |^p \, \mathrm{d}y \qquad \mbox{for}  \ z\geqslant R,
\end{align*} where $\widetilde{C}>0$ is a suitable positive constant. Finally, we observe that the power for $(y+R)$ in the last integral is equal to $-1$ if and only if $p=p_{\mathrm{Gla}}((1-k)n+2k)$. Repeating the same kind of computations as in Section \ref{Subsection proof Thm Eds}, we obtain the blow-up in finite time for $\widetilde{U}$ provided that $p\in(1,p_{\mathrm{Gla}}((1-k)n+2k)]$ and the corresponding upper bound estimates for the lifespan. We concluded the proof of Theorem \ref{Thm blow-up Tricomi negative}.

\section{Final remarks}

In the present paper, we proved blow-up results for the semilinear model
\begin{align}\label{final CP}
\begin{cases} \partial_t^2u-t^{-2k}\Delta u + \mu t^{-1} \partial_t u=|\partial_t u|^p, &  x\in \mathbb{R}^n, \ t>1,\\
u(1,x)=\varepsilon u_0(x), & x\in \mathbb{R}^n, \\ u_t(1,x)=\varepsilon u_1(x), & x\in \mathbb{R}^n,
\end{cases}
\end{align} when $\mu\in\{0,2\}$, provided that the Cauchy data fulfill suitable sign and support assumptions and that the exponent of the nonlinear term belongs to the following range $$1<p\leqslant p_{\mathrm{Gla}}((1-k)n+2k+\mu).$$ Due to the consistency of this result with other results known in the literature for special values of the parameters $k$ and $\mu$ (namely, for $k=0$ and/or $\mu=0$), we conjecture that the previous upper bound for $p$ could be the critical exponent for \eqref{final CP}. Nevertheless, in order to prove the validity of this conjecture the proof of the global (in time) existence of small data solutions in the case $p> p_{\mathrm{Gla}}((1-k)n+2k+\mu)$ is necessary. 

We point out that in the general case $\mu> 0$, the blow-up argument from Theorems \ref{Thm blow-up EdS} and \ref{Thm blow-up Tricomi negative} does not work sharply, especially for $\mu$ in the interval $(k,2-k)$. In the forthcoming paper \cite{HHP20}, we will study systematically the blow-up dynamic for \eqref{final CP} via a completely different approach.

\section*{Acknowledgments}

A. Palmieri is supported by the GNAMPA project 'Problemi stazionari e di evoluzione nelle equazioni di campo nonlineari dispersive'.

\end{document}